\documentclass[11pt]{article}
\usepackage{amsfonts}
\usepackage{amscd}
\usepackage{amsmath}
\usepackage{epsfig}
\usepackage{amsthm}
\usepackage{amssymb}
\usepackage{amsbsy}
\usepackage{latexsym}
\usepackage{eucal}
\usepackage{mathrsfs}
\usepackage[all, knot]{xy}
\xyoption{arc} \xyoption{color}\xyoption{line}

\newtheorem{definition}{Definition}
\newtheorem{theorem}{Theorem}

\newtheorem{lemma}{Lemma}
\newtheorem{prop}{Proposition}
\newtheorem{lem}{Lemma}

\newtheorem{Remark}{Remark}

\def\be{\begin{equation}}
\def\ee{\end{equation}}
\def\ben{\begin{displaymath}}
\def\een{\end{displaymath}}
\def\baa{\begin{eqnarray}}
\def\eaa{\end{eqnarray}}

\def\ba{\begin{array}}
\def\ea{\end{array}}

\renewcommand{\det}{\operatorname{det}}

\makeatletter
\@addtoreset{equation}{section}
\makeatother

\def\be{\begin{equation}}
\def\ee{\end{equation}}
\def\ben{\begin{displaymath}}
\def\een{\end{displaymath}}
\def\baa{\begin{eqnarray}}
\def\eaa{\end{eqnarray}}

\def\ba{\begin{array}}
\def\ea{\end{array}}

\def\2x2{{\left(\!\!\begin{array}{cc}a&b\\c&d\\\end{array}\!\!\right)}}

\newcommand{\dom}{\mbox{dom}}

\newcommand{\R}{{\mathbb{R}}}

\newcommand{\Nbb}{{\mathbb{N}}}
\newcommand{\Cbb}{{\mathbb{C}}}

\newcommand{\Ncal}{{\mathcal{N}}}

\newcommand{\Dhol}{\Delta_{\mathrm{hol}}}

\begin{document}
\title {DtN isospectrality, flat metrics with non-trivial holonomy and comparison formulas for determinants of Laplacian}

\author{Luc Hillairet\footnote{{\bf E-mail: luc.hillairet@univ-orleans.fr}}, Alexey Kokotov\footnote{{\bf E-mail: alexey.kokotov@concordia.ca}} }

\maketitle
\vskip0.5cm
\begin{center}
 MAPMO (UMR 7349 Universit\'e d'Orl\'eans-CNRS)
UFR Sciences, B$\hat{a}$timent de math\'ematiques
rue de Chartres,
BP 6759
45067 Orl\'eans Cedex 02
\end{center}

\vskip0.5cm
\begin{center}
Department of Mathematics and Statistics, Concordia
University, 1455 de Maisonneuve Blvd. West, Montreal, Quebec, H3G
1M8 Canada \end{center}

\vskip2cm
{\bf Abstract.} We study comparison formulas for $\zeta$-regularized determinants of self-adjoint extensions of the Laplacian on flat
conical surfaces of genus $g\geq 2$. The cases of trivial and non-trivial holonomy of the metric turn out to differ significantly.


\vskip2cm

\section{Introduction}
Let $X$ be a compact Riemann surface of genus $g\geq 2$ and let $D=P_1+\dots +P_{2g-2}$ be a positive divisor of degree $2g-2$ on $X$.
For simplicity, we consider only the generic situation: all the points
$P_k$ are assumed to be distinct. Then, due to the Troyanov theorem
(\cite{Troyanov}), there exists a unique (up to rescaling) conformal metric
${\bf m}$ on $X$ that is flat with conical singularities of angle $4\pi$ at $P_1, \dots, P_{2g-2}$.
If the divisor $D$ belongs to the canonical class then there also exists
a holomorphic one form $\omega$ on $X$ with simple zeros at $P_1, \dots,
P_{2g-2}.$ In that situation we thus obtain
\begin{equation}\label{metric}
{\bf m}=|\omega|^2\,,
\end{equation}
so that the Troyanov metric ${\bf m}$ has {\rm trivial holonomy}: i. e. the parallel transport along any loop
that avoids the conical singularities of $X$, maps any tangent vector
to $X$ to itself. Indeed, after the parallel transport along such a
loop, a tangent vector has turned for an angle $2\pi k$ with
$k\in{\mathbb Z}$. On the other hand, any flat metric with conical
singularities that has trivial holonomy is of the form (\ref{metric}).
It follows that the Troyanov metric ${\bf m}$ corresponding to the
divisor $D$ and conical angles $4\pi$ has trivial holonomy
if and only if the divisor $D$ belongs to the canonical class.

Consider now a general flat metric ${\bf m}$ with conical singularities of angle
$4\pi$, and define the Laplace operator $\Delta^{\bf m}: L^2(X, {\bf
  m})\to L^2(X, {\bf m})$ that is associated with it. This operator is
initially defined on the space $C^\infty_0(X\setminus \{P_1, \dots,
P_{2g-2}\})$ of smooth functions that vanish in a
neighbourhood of the conical points. The closure of $\Delta^{\bf m}$ is a symmetric
operator in $L^2(X, {\bf m})$ with deficiency indices $(6g-6, 6g-6).$
It thus has infinitely many self-adjoint extensions. Among these,
the Friedrichs extension, $\Delta_F$, is extensively studied. Explicit formulas for the (modified: i. e.
with zero modes excluded)  $\zeta$-regularized determinant ${\rm
  det}^*\Delta_F$  are found in \cite{KokKor} (in the case of trivial holonomy) and \cite{Kok} (in the general case).

The domain of the Friedrichs extension may include only functions that
are bounded near the conical singularities. As a consequence, when
the metric ${\bf m}$ has trivial holonomy,
the function
\begin{equation}\label{fun}u=\frac{\psi}{\omega}\,,\end{equation}
where $\omega$ is the holomorphic one form that defines the metric ${\bf
  m}=|\omega|^2$ and $\psi$ is an arbitrary holomorphic one form on X
not proportional to $\omega$,  belongs to $L^2(X, {\bf m})$, satisfies
the equation $\Delta^{\bf m}u=0$ in the classical sense outside the
singularities $P_1, \dots, P_{2g-2}$ but {\it does not} belong to the domain of the $\Delta_F$.
It turns out that there exists another natural s. a. extension of
$\Delta^{\bf m}$ that we denote by $\Delta_{\mathrm{hol}}$ whose
domain does contain the functions of the form (\ref{fun}).
Choosing a s.a. extensions is equivalent to prescribe a certain kind
of asymptotic behaviour near the conical points. For instance,
functions in $\dom(\Delta_F)$
have to be bounded but can have both holomorphic and antiholomorphic terms in the
asymptotic expansion near the conical points. On the contrary the
functions in $\dom(\Delta_{\mathrm{hol}})$ are not necessarily bounded
but may only have holomorphic terms in this expansion (hence the
name; see section \ref{sec:extensions}).

 In \cite{HK} it was found a general comparison formula
relating ${\rm det}^*\Delta_F$ to the determinant of any (regular) s. a. extension of the Laplacian for any flat conformal conical
metric on $X$ (with arbitrary conical angles).
The main player in this comparison formula is the so-called
$S$-matrix, $S(\lambda)$, of the conical surface $X$.
The goal of this paper is twofold. First we will apply the results of
\cite{HK} to relate the modified zeta-regularized determinant
$\det^*(\Delta_{\mathrm{hol}})$ with the determinant of the Friedrichs
extension. Then we will study in full detail the $S$-matrix in the
special case of metrics with conical angles $4\pi$: we will find explicit expressions for the
 entries of $S(\lambda)$ and their derivatives at $\lambda=0$ through invariant holomorphic objects related to the Riemann surface $X$.
 This will result in a nice explicit expression for the determinant of the
 holomorphic s.a. extension. The following theorem illustrates our
 result in genus $2.$

\begin{theorem}
Let $X$ be a Riemann surface of genus $2.$ For any two points $P_1$ and
$P_2$ on $X$, let ${\bf m}$ be the Troyanov flat metric with conical
singularities of angle $4\pi$ at $P_1$ and $P_2$ and let $\Delta_F$ and
$\Delta_{\mathrm{hol}}$ be the corresponding s.a. extensions of
$\Delta^{\bf m}.$\hfill \\

If the divisor $P_1+P_2$ {\it is not in the canonical class} or,
equivalently, if ${\bf m}$ {\it does not have trivial holonomy}, then
 \begin{equation}\label{showmain}
 {\rm det}^*\Delta_{hol}=c_2\left|
 \begin{matrix} B(P_1, P_1) \ \ B(P_1, P_2)\\
 B(P_2, P_1) \ \ B(P_2, P_2)
 \end{matrix}
     \right|{\rm det}^*\Delta_F\,,\end{equation}
 where $B(P_i, P_j)$ is the Bergman reproducing kernel for the
 holomorphic differentials on $X$ which is evaluated at $P_i,~P_j$
using distinguished holomorphic local parameters. The constant $c_2$ is
universal. \hfill \\

If the divisor $P_1+P_2$ {\it is in the canonical class} or,
equivalently, if ${\bf m}$ {\it has trivial holonomy}, then
 \begin{equation}\label{rav}
\det^*(\Delta_{\mathrm{hol}})\,=\,\det^*(\Delta_F).
\end{equation}
\end{theorem}

In the trivial holonomy case, it was previously observed by the first
author (\cite{H}) that the operators $\Delta_F$ and
$\Delta_{\mathrm{hol}}$ are DtN isospectral (see the precise definition below). This
 immediately implies that their $\zeta$-regularized determinants coincide.

 It should be noted that when the divisor $P_1+P_2$ enters the canonical class the matrix
 $$\left(
 \begin{matrix} B(P_1, P_1) \ \ B(P_1, P_2)\\
 B(P_2, P_1) \ \ B(P_2, P_2)
 \end{matrix}\right)
     $$
 becomes degenerate and the dimension of the kernel of
 $\Delta_{\mathrm{hol}}$ gets larger (by one unit).

 Relation (\ref{rav}) together with the comparison formula from \cite{HK} and the explicit information about the entries of $S(\lambda)$
obtained in the present paper, lead to identity (\ref{result}) below, which presents
the second main result of the paper. The latter identity seems to be
rather nontrivial and we failed to find any direct way to prove it.

 We put a certain effort to make it possible to read this paper independently of \cite{HK} (which is more technically involved):
 in our case of conical singularities of angle $4\pi$  the results and notations from \cite{HK} can be significantly simplified,
 we will briefly remind the reader all the constructions and main ideas from \cite{HK} using this opportunity
 to make the presentation more transparent.

{\bf Acknowledgements.}
 We thank D. Korotkin for numerous discussions. We are extremely grateful to K. Shramov and N. Tyurin for their help with the proof of Proposition 2; the second author also acknowledges
 useful discussions with P. Zograf, in particular, the question about the relation between ${\rm det}\,\Delta_F$ and
${\rm det}\,\Delta_{\mathrm{hol}}$  was raised by P. Zograf in one of these discussions.

 The research of AK was supported by NSERC.

\section{$S$-matrix of Euclidean surface with conical singularities}
In this section we briefly remind the setting of paper \cite{HK} adopting it to our (less general) situation.

\subsection{Self-adjoint extensions of Laplacian}\label{sec:extensions}
Let $X$ be a Riemann surface of genus $g\geq 2$ and let ${\bf m}$ be a flat conformal metric on $X$ with conical singularities $P_1, \dots, P_{2g-2}$ of conical
angles $4\pi$. Let $\Delta$ be the closure of the Laplace operator (corresponding to ${\bf m}$) with domain
$$C^\infty_0(X\setminus\{P_1, \dots, P_{2g-2}\})\subset L_2(X, {\bf m})$$
and let $\Delta^*$ be its adjoint operator. Let $\xi_k$ be the distinguished holomorphic local parameter in a vicinity of $P_k$, i . e.
$${\bf m}=4|\xi_k|^2|d\xi_k|^2$$
near $P_k$.  Near any $P_k$, any $u\in \dom (\Delta^*)$ has an
asymptotic expansion
\begin{equation}
\begin{split}
u(\xi_k, \bar \xi_k)& = a_k(u)\frac{i}{\sqrt{2\pi}}\log|\xi_k|\,+\,b_k(u) \frac{1}{2\sqrt{\pi}}
\frac{1}{\xi_k}+c_k(u)\frac{1}{2\sqrt{\pi}}\frac{1}{\bar \xi_k}\\
& + \frac{i}{\sqrt{2\pi}}d_k(u) \,+\, \frac{1}{2\sqrt{\pi}}e_k(u) \xi_k \,+\,\frac{1}{2\sqrt{\pi}}f_k(u)\bar{\xi_k}\,\\
& +o(|\xi|).
\end{split}
\end{equation}
the factors  $\frac{1}{2\sqrt{\pi}}$, etc are introduced in order to get the standard Darboux basis (see formula (\ref{symp}) below).

Let $u, v\in {\cal D}(\Delta^*)$, then Green formula implies the relation
\begin{equation}\label{symp}
\begin{split}
\Omega([u], [v])& :=\langle\Delta^*u, \bar v\rangle-\langle u,
\Delta^*\bar v\rangle\\
& =
\sum_{k=1}^{2g-2}X_k(u)\left(\begin{matrix}0\ \ \ -I_3\\I_3 \ \
    0 \end{matrix}\right){X_k(v)}^t\,,
\end{split}
\end{equation}
where $X_k(u)=(a_k(u), b_k(u), \dots, f_k(u))$ and
$\Omega$ is the symplectic  form on the factor space $\dom (\Delta^*)/\dom(\Delta)$.

The self-adjoint extensions of $\Delta$ are in one to one correspondence with the Lagrangian subspaces of $\dom (\Delta^*)/\dom(\Delta)$;
the Friedrichs extension corresponds to the lagrangian subspace
$$a_k(u)=b_k(u)=c_k(u)=0, \ k=1, \dots, 2g-2\,,$$
the holomorphic extension corresponds to the lagrangian subspace
$$a_k(u)=c_k(u)=f_k(u)=0, \ k=1, \dots, 2g-2\,.$$

\subsection{Special solutions and $S$-matrix}
For any $s\in \R,$ denote by $H^s_F(X) := \dom
(\Delta_F^{\frac{s}{2}}).$ Since $X$ is compact, the spectrum of $\Delta_F$ is discrete and
consists only in finite multiplicity eigenvalues. Let
$(\phi_\ell)_{\ell \geq 0}$ be a real orthonormal basis of
eigenfunctions for $\Delta_F$ and denote by $\lambda_\ell$ the
eigenvalue that corresponds to $\phi_\ell.$ This basis will be fixed
throughout the paper.

We define the linear functionals $\Lambda_{\xi_k}, \Lambda_{\bar{\xi}_k}$
  and $\Lambda_{0k}$ on $H^2_F(X)$ in such a way that any $u\in
  H^2_F(X) (:= \dom(\Delta_F))$ has the following asymptotic expansion
  near $P_k$ :
\[
u(\xi_k,\bar \xi_k) \,=\, \frac{i}{\sqrt{\pi}} \Lambda_{0k}(u) \,+\,
\frac{1}{2\sqrt{\pi}}\Lambda_{\xi_k}(u)\cdot
\xi_k\,+\,\frac{1}{2\sqrt{\pi}}\Lambda_{\bar \xi_k}(u)\cdot \bar \xi_k\,+\,o(|\xi_k|).
\]
Since $\Lambda_{\xi_k}$ are (continuous) linear functionals on
$H^2_F(X)$ it follows that the sequence
$\left( \frac{\Lambda_{\xi_k}(\phi_\ell)}{\lambda_\ell
    -\lambda}\right)_{\ell \geq 0}$ is in $\ell^2(\Nbb,\Cbb).$

Let $\lambda$ do not
belong to the spectrum $\Delta_F$.  We define $G_{\xi_k}$ by
\[
G_{\xi_k}:= \left( \Delta_F-\lambda\right )^{-1} \Lambda_{\xi_k}
\]
and $G_{\bar \xi_k},~G_{0k}$ by a similar formula.

By definition, we have
\begin{equation}\label{eq:expGk}
\forall k, ~ G_{\xi_k}(\cdot,\lambda) \,=\, \sum_{\ell \geq 0}
\frac{\Lambda_{\xi_k}(\phi_\ell)}{\lambda_\ell -\lambda}\phi_\ell.
\end{equation}
There are similar expressions for $G_{\bar \xi_k}.$

We recall that $G_{\xi_k}$ can be constructed in the following
way. Start with $F_{\xi_k}:= \frac{1}{2\sqrt{\pi}} \cdot \frac{1}{\xi_k}
\chi(\xi_k,\bar\xi_k)$ where $\chi$ is a cutoff function that is
identically one near $P_k.$ We then have
\[
G_{\xi_k} = F_{\xi_k} - (\Delta_F-\lambda)^{-1} (\Delta^*-\lambda)F_{\xi_k}.
\]
Thus, by construction, $G_{\xi_k}-F_{\xi_k}$ belongs to $H^2_F(X)$ and
so does $\partial_\lambda G_{\xi_k}.$ The latter has the following
(convergent in $L^2(X)$) expression in the chosen orthonormal basis:
\begin{equation}\label{eq:expdlG_k}
\partial_{\lambda} G_{\xi_k}(\cdot,\lambda) \,=\, \sum_{\ell \geq 0}
\frac{\Lambda_{\xi_k}(\phi_\ell)}{(\lambda_\ell -\lambda)^2}\phi_\ell.
\end{equation}

By evaluating the linear functionals $\Lambda$ on
$G_{\xi_k}-F_{\xi_k}$, we obtain a $(6g-6)\times(6g-6)$ matrix
$S(\lambda).$

In what follows we will often use the following convenient and
self-explanatory notation: the entries of $S$ will be denoted by
expressions of the type  $S^{\bar \xi_k \xi_j}(\lambda)$ that will
correspond to the coefficient of $\xi_j$ in the expansion of
$G_{\bar \xi_k}.$
Thus, by definition, near $P_k$ we will have
$$G_{\xi_k}=\frac{1}{2\sqrt{\pi}}
\frac{1}{\xi_k}\,+\,\frac{i}{\sqrt{2\pi}}S^{\xi_k0_k}(\lambda)\,+\, \frac{1}{2\sqrt{\pi}}S^{\xi_k \xi_k}(\lambda)\xi_k\,+\,\frac{1}{2\sqrt{\pi}}S^{\xi_k \bar\xi_k}(\lambda)\bar\xi_k+o(|\xi_k|)$$
as $\xi_k\to 0$ and near $P_j$ we will have
$$G_{\xi_k}=\frac{i}{\sqrt{2\pi}}S^{\xi_k0_j}(\lambda)+\frac{1}{2\sqrt{\pi}}S^{\xi_k \xi_j}(\lambda)\xi_j\,+\,\frac{1}{2\sqrt{\pi}}S^{\xi_k\bar\xi_j}(\lambda)\bar\xi_j+o(|\xi_j|)$$
as $\xi_j\to 0$.

It is convenient to represent this $S$-matrix with a block structure
$$S(\lambda)=\left(\begin{matrix}S_{00}(\lambda)\ S_{0a}(\lambda)\ S_{0h}(\lambda)\\
S_{a0}(\lambda)\ S_{aa}(\lambda) \ S_{ah}(\lambda)\\
S_{h0}(\lambda)\ S_{ha}(\lambda)\ S_{hh}(\lambda)
      \end{matrix}\right)\,,$$
where the indices $a$ and $h$ stand for holomorphic and
anti-holomorphic. Thus, for instance, the $S_{ah}$ block records
the coefficients of the anti-holomorphic part of the $G_{\xi_k}$:
\[
S_{ah}^{jk} \,=\, S^{\bar \xi_j \xi_k}
\]

 The $S$-matrix is holomorphic outside the spectrum of $\Delta_F$.
The blocks $S_{00}$, $S_{0a}$, $S_{0h}$, $S_{h0}$ and $S_{a0}$ blow up
 as $\lambda\to 0$, whereas the blocks $S_{aa}$, $S_{ah}$, $S_{ha}$
 and $S_{hh}$ are regular at $\lambda=0$.

In the next subsection we find explicit expressions for  $S_{aa}(0)$, $S_{ah}(0)$, $S_{ha}(0)$ and $S_{hh}(0)$.
 These expressions present the first new result of this paper.

\subsection{$S(0)$  through Schiffer and Bergman kernels}
\subsubsection{Some kernels from complex geometry}
Chose a marking for the Riemann surface $X$, i.e. a canonical basis $a_1, b_1, \dots, a_g, b_g$ of $H_1(X, {\mathbf Z})$. Let $\{v_1, \dots,
v_g\}$ be the basis of holomorphic differentials on $X$ that is normalized via
$$\int_{a_i}v_j=\delta_{ij}\,.$$
Then the matrix of $b$-periods of the marked Riemann surface $X$ is
defined as
$${\mathbb B}=[ b_{ij}]_{i,j\leq g},~~\mbox{with}~~b_{ij}:= \int_{b_i}v_j\,.$$

Let $W(\,\cdot\,,\,\cdot\,)$ be the canonical meromorphic
bidifferential on $X\times X$, (see \cite{Fay92}). Recall that it is
symmetric ($W(P,Q)=W(Q, P)$), and normalized by
$$\int_{a_i}W(\,\cdot\,, P)=0.$$
Recall also the following identities:
$$\int_{b_j}W(\,\cdot\,, P)=2\pi iv_j(P).$$
The only pole of the bidifferential $W$ is a double pole along the diagonal
$P=Q.$ In any holomorphic local parameter $x(P),$ the following
asymptotic expansion near the diagonal holds:
\begin{equation}\label{funH}W(x(P),
  x(Q))=\left(\frac{1}{(x(P)-x(Q))^2}+H(x(P),
    x(Q))\right)dx(P)dx(Q),\end{equation}
where
$$H(x(P), x(Q))=\frac{1}{6}S_B(x(P))+O(x(P)-x(Q)),$$
as $Q\to P$, with $S_B(\cdot)$ the Bergman projective connection.


The Schiffer bidifferential (or Schiffer kernel) is defined by
$${\cal S}(P, Q)=W(P, Q)-\pi\sum_{i, j}(\Im {\mathbb B})^{-1}_{ij}v_i(P)v_j(Q),$$
and the Schiffer projective connection, $S_{\mathrm{Sch}}$, by the asymptotic expansion
$${\cal S}(x(P), x(Q))=\left(\frac{1}{(x(P)-x(Q))^2}+\frac{1}{6}S_{\mathrm{Sch}}(x(P))+O(x(P)-x(Q))\right)dx(P)dx(Q)\,$$
It follows the equality
\begin{equation}\label{connection}
S_{\mathrm{Sch}}(x)=S_B(x)-6\pi \sum_{i, j}(\Im {\mathbf B})^{-1}_{ij}v_i(x)v_j(x)\,.\end{equation}

In contrast to the canonical meromorphic differential and the Bergman projective connection, the Schiffer bidifferential and the Schiffer projective connection are independent of the marking of the Riemann surface $X$.

We will also need the Bergman kernel which is the reproducing kernel for holomorphic
differentials on $X$. Its expression is given by
\begin{equation}\label{Berg}B(x, \bar x)=\sum_{ij}(\Im {\mathbb B})^{-1}_{ij}v_i(x)\overline{v_j(x)}\,.\end{equation}

\subsubsection{Harmonic functions on $X$}
Let $X$ be a compact Riemann surface of genus $g>1$ provided with flat conformal metric $m$ with $2g-2$ conical singularities at $P_1, \dots,
P_{2g-2}$ of conical angles $4\pi$. We do not assume that this metric has trivial holonomy: the divisor class $P_1+\dots +P_{2g-2}$
is not necessarily canonical.

The harmonic functions on $X\setminus \{ P_k\}$ which are square-integrable are uniquely determined by
their singular behaviour near the points $P_k.$ Thus, there exists a
unique harmonic function $H_k$ with the asymptotic behavior
$$H_k(\xi_k, \bar\xi_k)=\frac{1}{\xi_k}+a_{kk}+b_{kk}\xi_k+c_{kk}\bar \xi_k+o(|\xi_k|) \ \ \ {\rm at}\ \ P_k$$
$$H_k(\xi_j, \bar \xi_j)=a_{kj}+b_{kj}\xi_j+c_{kj}\bar \xi_j+o(|\xi_j|)\ \ {\rm at }\ \ P_j; \ j\neq k$$
(we remind the reader that $\xi_k$ denotes the distinguished local
parameter for ${\bf m}$ near $P_k$). Comparing with the definition of
$G_{\xi_k}$ we see that
\[
H_k = 2\sqrt{\pi} G_{\xi_k},
\]
and thus
\begin{gather*}
b_{kj}=S^{\xi_k \xi_j}(0)\\
c_{kj}=S^{\xi_k\bar\xi_j}(0).
\end{gather*}

Notice that  $S^{\bar\xi_k \bar\xi_j}(0)=\overline{S^{\xi_k
    \xi_j}(0)}$ and $S^{\bar \xi_k \xi_j}(0)=\overline{S^{\xi_k \bar
    \xi_j}(0)}$ and, therefore, the  $S(0)$ is completely known once we find $b_{kj}$ and $c_{kj}$.
\begin{prop}\label{scatt}
The following relations hold
\begin{gather*}
\begin{split}
S^{\xi_k\xi_k}&=-\frac{1}{6}S_{\mathrm{Sch}}(\xi_k)\Big|_{\xi_k=0}\\
S^{\xi_k\xi_j}&=-{\cal S}(P_k, P_j)\,, k\neq j\\
S^{\xi_k\bar \xi_j}&=-\pi B(P_k, P_j)\,,
\end{split}
\end{gather*}
where the values of (bi)differentials at the points $P_i$ are taken
w. r. t. the distinguished local parameter $\xi_i$.
\end{prop}
\begin{proof} Introduce the one forms $\Omega_k$ and $\Sigma_k$ on $X$:
$$\Omega_k=-W(\,\cdot\,, P_k)+2\pi i\sum_{\alpha, \beta}
\left(\Im {\mathbb B} \right)^{-1}_{\alpha \beta}\left\{\Im v_\beta(P_k)\right\}v_\alpha(\cdot)$$

$$\Sigma_k=-iW(\,\cdot\,, P_k)+2\pi i\sum_{\alpha, \beta}
\left(\Im {\mathbb B} \right)^{-1}_{\alpha \beta}\left\{\Re v_\beta(P_k)\right\}v_\alpha(\cdot)\,,$$
where
$$v_{\beta}(P_k):=v_\beta(\xi_k)|_{\xi_k=0}\,.$$

All the periods of the differentials $\Omega_k$ and $\Sigma_k$ are pure imaginary, therefore,
the following expression correctly defines a function $H_k$ on $X$
\begin{equation}\label{resh}H_k(Q)=\Re\left\{ \int_{P_0}^Q \Omega_k  \right\}-i\Re\left\{\int_{P_0}^Q\Sigma_k\right\}\,\end{equation}
where $P_0\neq P_k$ is an arbitrary base point.

Simple calculation shows that
$$H_k(\xi_k, \bar \xi_k)=\frac{1}{\xi_k}+{\rm const}+\left[-\frac{1}{6}S_B(P_k)+\pi\sum_{\alpha, \beta}(\Im{\mathbb B})^{-1}_{\alpha\beta}v_\alpha(P_k)v_\beta(P_k)\right]\xi_k+$$$$\left[-\pi\sum_{\alpha\beta}(\Im{\mathbb B})^{-1}_{\alpha\beta}\bar v_\alpha(P_k)v_\beta(P_k) \right]\bar\xi_k+o(|\xi_k|)$$
and
$$H_k(\xi_j, \bar \xi_j)={\rm const}+\left[-W(P_k, P_j)+\pi\sum_{\alpha\beta}(\Im{\mathbb B})^{-1}_{\alpha\beta}v_\beta(P_k)v_\alpha(P_j)
\right]\xi_j+$$$$\left[-\pi\sum_{\alpha\beta}(\Im{\mathbb B})^{-1}_{\alpha\beta}v_\beta(P_k)\bar v_\alpha(P_j)    \right]\bar\xi_j+o(|\xi_j|)$$
for $j\neq k$ which implies the proposition.
\end{proof}

\section{Bergman kernel and holomorphic one-forms.}

Let $B(\,\cdot\,,\,\cdot\,)$ be the Bergman reproducing kernel for holomorphic differentials (\ref{Berg}) on $X$, let $P_1, \dots, P_{2g-2}$ be 2g-2 distinct points of $X$ and let $\xi_k$ be a holomorphic local parameter at $P_k$.
Introduce the $(2g-2)\times(2g-2)$ matrix ${\cal B}$:
$${\cal B}=\left [ B(P_j, P_k)\right ]_{ j,k=1, \dots, 2g-2}\,,$$
where the value of $B$ at $P_k$ is taken with respect to the local
parameter $\xi_k$. We observe that the rank of $\mathcal{B}$ is
independent of the choice of the local parameter.

According to proposition \ref{scatt} we have
\[
{\cal B} = -\frac{1}{\pi} S_{ha}(0).
\]

It follows that the $S$ matrix knows properties of the Riemann surface
that are encoded in $\mathcal{B}.$ For instance we have the following Proposition.

\begin{prop} The following two statements are equivalent
\begin{itemize}
\item[(A)] The divisor $P_1+\dots+P_{2g-2}$  coincides with the divisor of some holomorphic one-form $\omega.$
\item[(B)] ${\rm rank}\, {\cal B}<g.$
\end{itemize}
\end{prop}

Before proving this proposition, we recall the
\begin{definition}
On a Riemann surface $X$ of genus $g$, a divisor $Q_1+Q_2+\dots Q_g$ is called {\em
  special} when
\[
i(Q_1+\dots+Q_g):={\rm dim}\{\omega - {\rm holomorphic\  one\  form}:
\ (\omega)>Q_1+\dots +Q_g \ \  {\rm or}\ \ \omega=0\}>0\,.
\]
Equivalently, the divisor $Q_1+\dots Q_g$ is special iff
\begin{equation}\label{det0}{\rm det} \left [ v_j(Q_k)\right ]_{1\leq
    j,k\leq g}=0.\end{equation}
\end{definition}

\begin{proof}
\noindent ${\bf (A)\Rightarrow (B)}$ \\
Since the positive divisor $P_1+\dots +P_{2g-2}$ is in the canonical class, for any $i_1, \dots, i_g$; $1\leq i_1<i_2\dots<i_g\leq 2g-2$ the divisor $P_{i_1}+\dots+P_{i_g}$ is special.

The matrix ${\cal B}$ is the Gram matrix of the $2g-2$ vectors from ${\mathbb C}^{g}$:
$${\bf V_1}=(v_1(P_1), v_2(P_1), \dots, v_g(P_1))$$
$$...............................$$
$${\bf V_{2g-2}}=(v_1(P_{2g-2}), v_2(P_{2g-2}), \dots, v_g(P_{2g-2}))$$
with respect to the (nondegenerate) Hermitian product in ${\mathbb C}^g$
$$\langle{\bf V}, {\bf W}\rangle =\sum_{\alpha, \beta=1}^g\Im {\mathbb B}_{\alpha \beta}^{-1}V_\alpha\bar W_\beta\,.$$
Thus, ${\rm rank}\ {\cal B}\leq g$ and if ${\rm rank}\ {\cal B}= g$,
then, according to Principal Minor Theorem for Hermitian matrices
(see, e. g., \cite{Egan}), there exists a nonzero {\rm principal}
minor of ${\cal B}$ of order $g$. This contradicts (\ref{det0}).\hfill \\

\noindent{${\bf (B)\Rightarrow (A)}$:}
 We will use the following elementary fact from linear algebra:

Let $E$ be a vector space of dimension $g$ and let $E^*$ be its dual space. Let $k\geq g$ and $\{L_j\}_{j=1, \dots, k}$ be a collection of linear forms on $E$.  Then
$${\rm Span}\,_{j=1, \dots, k}L_j=E^* \Leftrightarrow \cap_{j=1}^k {\rm Ker}\,L_j=\{0\} \Leftrightarrow $$$$\exists i_1, \dots, i_g\,:\, L_{i_1}, \dots, L_{i_g} {\rm are \ \ linear\ \ independent}\,.$$

Let $E$ be the ($g$-dimensional) space of holomorphic one forms on $X$, and let $L_j(v)=v(P_j)$, where $v$ is a holomorphic one-form, and $v(P_j)$ is its value at $P_j$ with respect to the distinguished local parameter of ${\bf m}$ at $P_j$. If ${\rm rank}\, {\cal B}<g$ then no set of $g$ forms $L_j$ is independent and, therefore, the intersection of all ${\rm Ker}\, L_j$ is not trivial. Thus, there exists a nonzero holomorphic one form vanishing at all $P_j$.
\end{proof}

\begin{Remark}\footnote{We thank K. Shramov and N. Tyurin  for this remark}
The second part of the proof can be simplified when $X$ is
non-hyperelliptic by  using that, in this case $X\ni P\to (v_1(P), \dots, v_g(P))\in{\mathbb C}P^{g-1}$ is an embedding.
\end{Remark}

\begin{Remark}  The $S$-matrix knows whether the Troyanov metric has trivial holonomy or not.
\end{Remark}

\section{Dirichlet-to-Neumann isospectrality and Bergman kernel}

\subsection{Comparison formula for determinants and Bergman kernel}

The following Proposition is essentially a reformulation of Theorem 1
from \cite{HK} for the case of the holomorphic extension $\Delta_{\mathrm{hol}}$.
Since  this reformulation is not completely immediate, we give here some details on the proof.
\begin{prop}\label{DD} Let
$$T(\lambda):=S^{ha}(\lambda)=\left(
\begin{array}{cccc}
S^{\xi_1\bar \xi_1}(\lambda) & S^{\xi_1, \bar \xi_2}(\lambda) &\cdots &
 S^{\xi_1, \bar \xi_{2g-2}}(\lambda)\\
\vdots & \vdots & \cdots & \vdots \\
S^{\xi_{2g-2}\bar \xi_1}(\lambda) & S^{\xi_{2g-2}, \bar
  \xi_2}(\lambda) & \cdots & S^{\xi_{2g-2}, \bar \xi_{2g-2}}(\lambda)
\end{array}
\right)\,.$$
There exists a constant $C_g$ that depends only on the genus such that
the following comparison formula for the $\zeta$-regularized
determinants holds:
$${\rm det}(\Dhol-\lambda)=C_g\,{\rm detT} (\lambda){\rm det}(\Delta_F-\lambda)\,.$$
\end{prop}

To get Proposition \ref{DD} one needs the following reformulation of Proposition 5.3 from \cite{HK}
for the case of the holomorphic extension.
\begin{lemma} The following relation holds true
\begin{equation}\label{newver}
{\rm Tr}((\Dhol-\lambda)^{-1}-(\Delta_F-\lambda)^{-1})=-{\rm Tr}\left(T(\lambda)^{-1}\frac{d}{d\lambda}T(\lambda)
\right)\,.\end{equation}
\end{lemma}

\begin{proof} As in \cite{HK} the key idea is to make use of the M. G. Krein theory for the difference of the resolvents
of two self-adjoint extensions of a symmetric operator with finite
deficiency indices. Using this theory, for any $f\in L^2(X)$ there
exists a collection $(x_k)_{k=1,\dots,2g-2}$ such that
\begin{equation}\label{AA}
(\Dhol-\lambda)^{-1}f=(\Delta_F-\lambda)^{-1}f\,+\, \sum_{k=1}^{2g-2}
x_k G_{\xi_k}(\cdot, \lambda)
\end{equation}
For each $j,$ considering the coefficient of $\bar{\xi}_j$ leads to
the following equation :
\[
\forall j, ~ 0\,=\, \Lambda_{\bar
  \xi_j}\left((\Delta_F-\lambda)^{-1}f\right)\,+\,\sum_{k=1}^{2g-2}
x_k S^{\xi_k \bar{\xi}_j}.
\]
Since $S^{\xi_k \bar{\xi}_j}= T_{kj},$ by inverting this relation we obtain that
\[
\forall k,~ x_k \,=\,-\sum_{j=1}^{2g-2}  \Lambda_{\bar
  \xi_j}\left((\Delta_F-\lambda)^{-1}f\right) T^{-1}_{jk}
\]

It follows that by taking the trace we obtain
\begin{equation}
\begin{split}
{\rm Tr}((\Dhol-\lambda)^{-1}-(\Delta_F-\lambda)^{-1})& \,=\,-\sum_{j,k=1}^{2g-2}\sum_{\ell
  \geq 0} \Lambda_{\bar
  \xi_j}\left((\Delta_F-\lambda)^{-1} \phi_\ell \right)
T^{-1}_{jk}\langle \phi_\ell, G_{\xi_k}\rangle \\
&=\,-\sum_{j,k=1}^{2g-2} T^{-1}_{jk} \sum_{\ell \geq 0} (\lambda_\ell
-\lambda)^{-2} \Lambda_{\bar
  \xi_j}(\phi_\ell)\Lambda_{\xi_k}(\phi_\ell)
\end{split}
\end{equation}
Using equation \eqref{eq:expdlG_k}, we get
\[
\sum_{\ell \geq 0} (\lambda_\ell
-\lambda)^{-2} \Lambda_{\bar
  \xi_j}(\phi_\ell)\Lambda_{\xi_k}(\phi_\ell) \,=\, \Lambda_{\bar
  \xi_j}\left( \partial_\lambda G_{\xi_k}\right)
\,=\, \partial_{\lambda} S^{\xi_k \bar \xi_j}\,=\, \partial_\lambda T_{kj}.
\]
The Lemma thus follows.
\end{proof}

This lemma implies Proposition \ref{DD} by performing exactly the same
contour integration as in \cite{HK} (from that point, the remaining
part of the proof can be repeated verbatim).

\subsection{DtN isospectrality of $\Delta_F$ and $\Delta_{\mathrm{hol}}$}

Let the divisor $P_1+\dots+P_{2g-2}$ be in the canonical class. Then the corresponding Troyanov metric (all the conical angles are  equal to $4\pi$)
has trivial holonomy and is given by $|\omega|^2$, where $\omega$ is a holomorphic one-form with simple zeros at $P_1, \dots, P_{2g-2}$.
Let $z(P)=\int_{P_0}^P \omega$ with some base point $P_0$. Then $z(P)$ can be taken as holomorphic local parameter on $X$ in a vicinity
 of any point of $X\setminus\{P_1, \dots, P_{2g-2}\}$.
 Consider the operator
 $$D_{z}: L^2(X, |\omega|^2)\supset C^\infty_0(X\setminus\{P_1, \dots, P_{2g-2}\}\ni u\rightarrow u_z\in L^2(X, |\omega|^2)$$

\begin{Remark} Let $\partial_z$ be the standard Cauchy-Riemann operator $\partial_z: C^\infty(X)\to \Lambda^{1, 0}(X)$.
Then $D_z=\frac{1}{\omega}\partial_z$.
\end{Remark}

The operator $D_{z}$ is closable. Denote its closure again by $D_z$. The following observation belongs to the first author
\cite{H}.
\begin{prop}\label{iso} One has the following identifications of the self-adjoint operators
$$D_z^*D_z=\frac{1}{4}\Delta_F\,,$$
$$D_zD_z^*=\frac{1}{4}\Dhol\,.$$
\end{prop}

\begin{proof}
By general considerations, both operators are self-adjoint extensions
of the Laplace operator on functions that vanish near the conical
points. It thus suffices to prove that the domains coincide. For the
first operator, it follows by remarking that  $\dom(D_z^*D_z)\subset
\dom(D_z).$
For the second operator it follows by remarking that any function in
$\dom(\Delta^*)$ with holomorphic behaviour near the conical points is
in $\dom(D_zD_z^*).$
\end{proof}

From Proposition \ref{iso} it immediately follows that the operators $\Delta_F$ and $\Delta_N$ are DtN isospectral, that is
$${\rm dim\,}{\rm Ker\,}(\Delta_F-\lambda)= {\rm dim\,}{\rm Ker\,}(\Dhol-\lambda)$$
for any $\lambda\neq 0$.

To compare the spectral determinants of both extensions, it thus
suffices to compute the dimension of their kernels. The kernel of
$\Delta_F$ only consists of the constant functions.

\begin{lem} For metrics with trivial holonomy one has the equality
$${\rm dim}\,{\rm Ker}\, \Dhol=g\,.$$
\end{lem}
\begin{proof}
Denote by $X_\epsilon$ the surface $X$ where $\epsilon$-disks around
$P_k$ have been deleted. Then, for any $u\in \dom(\Dhol)$ we have .
\begin{equation*}
\begin{split}
-\frac{1}{4}\langle u, \Delta u\rangle & =\lim_{\epsilon\to 0}\int_{X_\epsilon}\bar
u \partial_z\partial_{\bar z}u|dz|^2\\
&= \langle\partial_{\bar z}u, \partial_{\bar z}u\rangle \\
& \,+\,\lim_{\epsilon \to 0}\sum_{k=1}^{2g-2}\oint_{|\xi_k|=\sqrt{\epsilon}}
(A_k/\xi_k+\dots)\frac{1}{\bar \xi_k}\partial_{\bar
  \xi_k}(A_k/\xi_k+B+C\xi_k+O(|\xi_k|^2))\bar\xi_kd\bar\xi_k\\
& = \langle \partial_{\bar z}u, \partial_{\bar z}u\rangle.
\end{split}
\end{equation*}
It follows that if $u\in {\rm Ker}\, \Delta_{\mathrm{hol}}$ then
$\partial_{\bar z}u=0.$ Thus, the one form $u\omega$ is holomorphic
and, therefore,
$u\in {\rm Span}\{v_1/\omega, \dots, v_g/\omega\}$.
Conversely, for any holomorphic one-form $v$ the meromorphic function
$v/\omega$ is seen to be in $\dom(\Dhol)$ and satisfies $\Delta^*
v/\omega =0.$ We obtain
\[
\ker(\Dhol) \,=\,{\rm Span}\{v_1/\omega, \dots, v_g/\omega\}.
\]
\end{proof}

\begin{Remark}
The needed equality also follows form the DtN isospectrality of $\Delta_{hol}$ and $\Delta_F$ and Remark
5.11 from \cite{HK}. Indeed, the latter implies the relation
$$\zeta(0, \Delta_{hol}-\lambda)=\zeta(0, \Delta_F-\lambda)+g-1$$
between the values of the operator zeta-functions of $\Delta_{\mathrm{hol}}-\lambda$ and $\Delta_F-\lambda$ at $s=0$.
\end{Remark}

The DtN isospectrality of $\Dhol$ and $\Delta_F$ (see \cite{H},
Theorem 4.8) then implies

$${\rm det}(\Dhol-\lambda)=\lambda^{g-1}{\rm det}(\Delta_F-\lambda)$$

and, therefore,

\begin{equation}\frac{{\rm det} \,T(\lambda)}{\lambda^{g-1}}=C_g\end{equation}

and

\begin{equation}\label{IDENTITY}
C_g=\frac{1}{(g-1)!}\left(\frac{d}{d\lambda}\right)^{g-1}{\rm det}\,T(\lambda)\Big|_{\lambda=0}\,,
\end{equation}

For any integer $M,$ denote by $\Ncal_{M}$ the set of $(2g-2)$-tuples
$\mathbf{n}:=(n_1,\cdots,n_{2g-2})$ of non-negative integers such that
$n_1+\cdots +n_{2g-2}=M.$
The preceding identity can thus be reformulated as
\begin{equation}\label{result}
C_g\,=\,\sum_{{\mathbf n}\in \Ncal(g-1)}\frac{1}{n_1!\dots n_{2g-2}!}
\left|
\begin{array}{ccc}
\left(\frac{d}{d\lambda}\right)^{n_1}S^{\xi_1 \bar \xi_1}(0) & \cdots & \left(\frac{d}{d\lambda}\right)^{n_{2g-2}}S^{\xi_1 \bar \xi_{2g-2}}(0)\\
\vdots & \cdots & \vdots \\

\left(\frac{d}{d\lambda}\right)^{n_1}S^{\xi_{2g-2} \bar \xi_1}(0) & \cdots & \left(\frac{d}{d\lambda}\right)^{n_{2g-2}}S^{\xi_{2g-2} \bar \xi_{2g-2}}(0)
\end{array}
\right|\,
\end{equation}
and $C_g$ is an absolute constants that depends on the genus only.

We now give alternative expressions for the coefficients in the
preceding determinant.

\begin{lemma}\label{lem} For all $n\geq 1$ and all $j,k,$ the following generalization
  of the formula (4.2) from \cite{HKK2014}
holds:
\begin{equation}
(\partial_\lambda)^nS^{\xi_k \bar
    \xi_j}(\lambda)=n!\int_X[(\Delta_F-\lambda)^{1-n}G_{\xi_k}]G_{\bar
    \xi_j} dS.
\end{equation}
\end{lemma}

\begin{proof}
We start from \eqref{eq:expdlG_k}, apply $\Lambda_{\bar \xi_j}$ and
differentiate $n-1$ times to obtain
\begin{equation*}
\begin{split}
(\partial_\lambda)^nS^{\xi_k \bar
    \xi_j}(\lambda)& \,=\, n!\sum_{\ell \geq 0}
  (\lambda_\ell-\lambda)^{-1-n}\Lambda_{\xi_k}(\phi_\ell)
  \Lambda_{\bar \xi_j}(\phi_\ell)\\
& \,=\, n!\sum_{\ell \geq 0}
  (\lambda_\ell-\lambda)^{1-n}\frac{\Lambda_{\xi_k}(\phi_\ell)}{\lambda_\ell-\lambda}
  \frac{\Lambda_{\bar \xi_j}(\phi_\ell)}{\lambda_\ell -\lambda}\\
&\,=\, n!\int_X[(\Delta_F-\lambda)^{1-n}G_{\xi_k}]G_{\bar
    \xi_j} dS.
\end{split}
\end{equation*}
\end{proof}

From the expansion in the eigenfunctions basis, we now observe that
$G_{\xi_k}\bot 1$ and is analytic at $0.$
It follows that
\begin{equation}\label{N1}
G_{\xi_k}(\cdot; \lambda=0)={\cal H}_k:=\frac{1}{2\sqrt{\pi}}\big(H_k-\frac{1}{A}\int_XH_k\big), \end{equation}
where
$H_k$ is explicitly given by (\ref{resh}).

The operator
$$\Delta_F^{-1}:1^\bot \to 1^\bot$$
has an integral kernel which is given by the Green function
\begin{equation}\label{N2}G(x, y)=\frac{1}{2\pi A^2}\int_X\int_X \Re \left(\int_b^x\Omega_{y-a}\right)dS(a)dS(b)\,\end{equation}
where
$$\Omega_{a-b}(z)=\int_a^b W(z, \cdot)-2\pi i\sum_{\alpha \beta}(\Im{\mathbb B})^{-1}_{\alpha \beta}v_\alpha(z)\Im\int_a^bv_\beta$$
(see, \cite{Fay92}, formula (2.19)).
Thus, by letting $\lambda$ go to $0$ we obtain the following explicit
expression for all the terms in the r. h. s. of (\ref{result}) :
\begin{equation}\label{N3}\left(\frac{d}{d\lambda}\right)^{n_k}_{|{\lambda=0}}S^{\xi_l \bar \xi_k}(\lambda)=
\int_X [(\Delta_F|_{1^\bot})^{1-n_k}{\cal H}_l ]\overline{{\cal H}_k}\,\end{equation}
where
\begin{equation}\label{N4}(\Delta_F|_{1^\bot})^{1-n_k}{\cal H}_l(x)=\end{equation}

$$\int_XG(x, x_1)\int_XG(x_1, x_2)\dots\int_XG(x_{n_k-1}, y){\cal H}_l(y)dS(y)dS(x_{n_k-1})\dots dS(x_1)\,.$$

\section{Metrics with nontrivial holonomy}

\begin{lem}\label{lastlem} Let $P_1, \dots, P_{2g-2}$ be points of $X$ such that the divisor
$D=P_1+\dots+P_{2g-2}$ is not canonical.
(According to Proposition 2,  this condition is equivalent to the condition
${\rm rank}\,{\cal B}=g$.)
Then we have
$${\rm dim}\,{\rm Ker}\, \Dhol=g-1\,.$$
\end{lem}

\begin{proof}
 Let $u\in {\rm Ker}\, \Dhol$. Then, similarly to the proof of Lemma 1, one can show that $u$ is holomorphic in $X\setminus\{P_1, \dots
P_{2g-2}\}$.
According to Corollary of Theorem III.9.10 from \cite{Farkas}, $D$ is the divisor of a holomorphic multi-valued Prym differential $\omega$
with nontrivial (multiplicative) unitary monodromy $T_{a_\alpha}$, $T_{b_\alpha}$ along basic cycles $a_\alpha$, $b_\alpha$:
$$T_{a_\alpha}\omega=e^{i r_\alpha}\omega; T_{b_\alpha}\omega=e^{i s_\alpha}\omega; \ \ r_\alpha, s_\alpha\in {\mathbb R}\,.$$

Then $u\omega$ is a Prym differential with the same monodromy. Since the space of Prym differentials with a given nontrivial unitary
monodromy has dimension $g-1$ (see \cite{Weyl}, p. 147), one gets
$$ {\rm dim}{\rm ker}\Delta_{\mathrm{hol}}\leq g-1\,.$$
The inequality
$$ {\rm dim}{\rm ker}\Delta_{\mathrm{hol}}\geq g-1\,$$
is obvious: take $\{\omega_1, \dots, \omega_{g-1}\}$ a basis of the space of
Prym differentials with the same monodromy as $\omega$, then the
functions $\omega_1/\omega, \dots, \omega_{g-1}/\omega$ are linearly
independent, belong to $\dom(\Delta_{\mathrm{hol}})\cap \ker \Delta^*$.
 \end{proof}

Lemma \ref{lastlem} and Proposition \ref{DD} imply the following Theorem.

\begin{theorem} Let the divisor $P_1+\dots+P_{2g-2}$ do not belong to the
canonical class. Then
\begin{multline*}
\frac{{\rm det}\Dhol}{{\rm det}\Delta_F}\,= \\
\,c_g \cdot\sum_{{\mathbf n}\in \Ncal(g-2)}
\frac{1}{n_1!\dots n_{2g-2}!}\left |
\begin{array}{ccc}
\left(\frac{d}{d\lambda}\right)^{n_1}S^{\xi_1 \bar \xi_1}(0) &\dots
  & \left(\frac{d}{d\lambda}\right)^{n_{2g-2}}S^{\xi_1 \bar \xi_{2g-2}}(0)\\
\vdots & \cdots & \vdots \\
\left(\frac{d}{d\lambda}\right)^{n_1}S^{\xi_{2g-2} \bar \xi_1}(0) &
                                                                    \dots &
\left(\frac{d}{d\lambda}\right)^{n_{2g-2}}S^{\xi_{2g-2} \bar \xi_{2g-2}}(0)
\end{array}
\right|\,,
\end{multline*}
where all the terms at the right are explicitly given by (\ref{N1},
\ref{N2}, \ref{N3}, \ref{N4}) and the absolute constant $c_g$ depends
only on the genus $g$.
In particular, in genus 2 one gets relation (\ref{showmain}).
\end{theorem}

\end{document}